\documentclass{birkjour}

\usepackage[utf8]{inputenc}
\usepackage[T1]{fontenc}
\usepackage[english]{babel}
\usepackage{mathtools,amsmath,amssymb,amsfonts,amsthm,mathrsfs}
\mathtoolsset{showonlyrefs}
\usepackage{enumitem}

\usepackage{geometry}
\usepackage{graphicx,color}
\usepackage{tikz,pgfplots,pgf}
\usepackage[caption=false]{epsfig}
\usepackage{subcaption}
\usepackage[font=small]{caption}

\newtheorem{theorem}{Theorem}[section]

\newtheorem{lemma}{Lemma}[section]

\theoremstyle{definition}
\newtheorem{remark}{Remark}[section]

\numberwithin{equation}{section}

\begin{document}

\title[Pairs of solutions for a Minkowski-curvature indefinite Neumann problem]{Pairs of positive radial solutions for a \\Minkowski-curvature Neumann problem \\with indefinite weight}

\author[A.~Boscaggin]{Alberto Boscaggin}

\address{
Department of Mathematics ``Giuseppe Peano'', University of Torino\\
Via Carlo Alberto 10, 10123 Torino, Italy}

\email{alberto.boscaggin@unito.it}

\author[G.~Feltrin]{Guglielmo Feltrin}

\address{
Department of Mathematics, Computer Science and Physics, University of Udine\\
Via delle Scienze 206, 33100 Udine, Italy}

\email{guglielmo.feltrin@uniud.it}

\thanks{Work written under the auspices of the Grup\-po Na\-zio\-na\-le per l'Anali\-si Ma\-te\-ma\-ti\-ca, la Pro\-ba\-bi\-li\-t\`{a} e le lo\-ro Appli\-ca\-zio\-ni (GNAMPA) of the Isti\-tu\-to Na\-zio\-na\-le di Al\-ta Ma\-te\-ma\-ti\-ca (INdAM). The first author is supported by INdAM-GNAMPA project ``Il modello di Born--Infeld per l'elettromagnetismo nonlineare: esistenza, regolarit\`{a} e molteplicit\`{a} di soluzioni''.
\\
\textbf{Preprint -- December 2019}} 



\date{}

\dedicatory{}

\begin{abstract}
We prove the existence of a pair of positive radial solutions for the Neumann boundary value problem
\begin{equation*}
\begin{cases}
\, \mathrm{div}\,\Biggl{(} \dfrac{\nabla u}{\sqrt{1- | \nabla u |^{2}}}\Biggr{)} + \lambda a(|x|)u^p = 0, & \text{in $B$,} \\
\, \partial_{\nu}u=0, & \text{on $\partial B$,}
\end{cases}
\end{equation*}
where $B$ is a ball centered at the origin, $a(|x|)$ is a radial sign-changing function with
$\int_B a(|x|)\,\mathrm{d}x < 0$, $p>1$ and $\lambda > 0$ is a large parameter.
The proof is based on the Leray--Schauder degree theory and extends to a larger class of nonlinearities.
\end{abstract}

\maketitle

\section{Introduction}\label{section-1}

Motivated by its interest in Differential Geometry and General Relativity, in the last years a great deal of research has been devoted to the study of boundary value problems associated with the Minkowski-curvature equation 
\begin{equation}\label{eq-inizio}
\mathrm{div}\,\Biggl{(} \dfrac{\nabla u}{\sqrt{1- | \nabla u |^{2}}}\Biggr{)} + f(x,u) = 0, \qquad x \in \Omega \subseteq \mathbb{R}^N,
\end{equation}
both in the ODE case ($N=1$) and in the PDE case ($N \geq 2$); see, for instance, \cite{Az-14,BeJeTo-13,BeMa-07,BoCoNo-pp,BoGa-19ccm,CoObOmRi-13,Ma-13} and the references therein.
Quite frequently, it is assumed that \eqref{eq-inizio} admits the trivial solution (i.e.~$f(x,u) \equiv 0$) and existence/multiplicity of positive solutions of \eqref{eq-inizio} is investigated.

On this line of research, in the recent paper \cite{BoFe-PP} we searched for positive periodic solutions (both harmonic and subharmonic) to the parameter-dependent equation
\begin{equation}\label{eq-ode}
\Biggl{(} \frac{u'}{\sqrt{1-(u')^{2}}}\Biggr{)}' + \lambda  a(x) g(u) = 0,
\end{equation}
where $\lambda > 0$, $a(x)$ is a $T$-periodic sign-changing (i.e.~indefinite) weight function and  
$g(u)$ is a nonlinear term satisfying $g(0) = 0$.
Among other results, we proved therein that a two-solution theorem holds for the $T$-periodic boundary value problem associated with equation \eqref{eq-ode}: more precisely, for weight functions $a(x)$ satisfying the mean value condition $\int_{0}^{T} a(x)\,\mathrm{d}x < 0$
and for a large class of nonlinear terms $g(u)$ which are superlinear at zero (namely, $g(u)/u \to 0$ for $u \to 0^{+}$), two positive $T$-periodic solutions of \eqref{eq-ode} exist, whenever the parameter $\lambda$ is large enough (see \cite[Theorem~3.1]{BoFe-PP} for the precise statement of this result). We refer the reader to the introduction of \cite{BoFe-PP} for several comments about this solvability pattern, arising as a result of a delicate interplay between the behaviors of the nonlinear differential operator driving equation \eqref{eq-ode} and the nonlinear term $a(x)g(u)$ when $u \to +\infty$.

The technical tool used in \cite{BoFe-PP} to prove the above mentioned result is topological degree theory in Banach spaces, along a line of research started in \cite{FeZa-15jde} and later developed and applied in several different situations (cf.~\cite{Fe-18}), always dealing with nonlinear BVPs for semilinear equations of the type
\begin{equation}\label{eq-ode2}
u'' + q(x)g(u) = 0,
\end{equation}
where $q(x)$ is an indefinite weight function. As well known, the first step within this approach is the formulation of the differential equation as a nonlinear functional equation in a Banach space: this can be done in a standard way when considering equation \eqref{eq-ode2}, since the differential operator $Lu = -u''$ is a (linear) Fredholm operator of index zero. Then, depending on the invertibility/non-invertibility of $L$ (and, hence, on the boundary conditions), either classical Leray--Schauder degree theory or Mawhin's coincidence degree theory apply. 

As far as the strongly nonlinear equation \eqref{eq-ode} is concerned, different strategies can be followed to achieve this goal. In \cite{BoFe-PP}, we chose to write \eqref{eq-ode} as the equivalent planar system
\begin{equation*}
u' = \frac{v}{\sqrt{1+v^2}}, \qquad v' = -\lambda a(x)g(u),
\end{equation*}
in order to directly apply coincidence degree theory in the product space, as proposed in the recent paper \cite{FeZa-17tmna}.
This approach, which looks very natural when dealing with the periodic problem, has the drawback of not being suited for other boundary conditions. In particular, in spite of the well-known strong analogies existing in this setting between the periodic and the Neumann boundary value problem (see, for instance, \cite{BoZa-15}), the possibility of proving the Neumann counterpart of the result in \cite{BoFe-PP} is not discussed therein.

The aim of this brief paper is to provide a positive answer to this question. 
More generally, we deal with the Neumann boundary value problem for the PDE version of equation \eqref{eq-ode}, namely
\begin{equation}\label{eq-main-pde}
\begin{cases}
\, \mathrm{div}\,\Biggl{(} \dfrac{\nabla u}{\sqrt{1- | \nabla u |^{2}}}\Biggr{)} + \lambda a(|x|) g(u) = 0, & \text{in $B$,} \\
\, \partial_{\nu}u=0, & \text{on $\partial B$,}
\end{cases}
\end{equation}
where $B$ is a ball of the $N$-dimensional Euclidean space and $a(\vert x \vert)$ is a (sign-changing) radial weight function. 

In this framework, we prove the following two-solution theorem for positive radial solutions of \eqref{eq-main-pde}.

\begin{theorem}\label{th-main}
Let $N \geq 1$ be an integer and let $B \subseteq \mathbb{R}^N$ be an open ball of center the origin and radius $R>0$.
Let $a \colon \mathopen{[}0,R\mathclose{]} \to \mathbb{R}$ be an $L^{1}$-function such that
\begin{itemize}[leftmargin=30pt,labelsep=12pt,itemsep=5pt]
\item [$(a_{*})$]
there exist $m\geq 1$ closed and pairwise disjoint intervals $I^{+}_{1},\ldots,I^{+}_{m}$ in $\mathopen{[}0,R\mathclose{]}$ such that
\begin{align*}
&\qquad\qquad a(r)\geq0, \; \text{ for a.e.~$r\in I^{+}_{i}$,} \quad a\not\equiv0 \; \text{ on $I^{+}_{i}$,} \quad \text{for $i=1,\ldots,m$,} \\
&\qquad\qquad a(r)\leq0, \; \text{ for a.e.~$r\in \mathopen{[}0,R\mathclose{]}\setminus\bigcup_{i=1}^{m}I^{+}_{i}$;}
\end{align*}
\item [$(a_{\#})$] $\displaystyle \int_{B} a(|x|) \,\mathrm{d}x < 0$.
\end{itemize}
Let $g \colon \mathopen{[}0,+\infty\mathclose{[} \to \mathopen{[}0,+\infty\mathclose{[}$ be a continuous function satisfying 
\begin{itemize}[leftmargin=30pt,labelsep=12pt,itemsep=5pt]
\item [$(g_{*})$] $g(0)=0$ and $g(u)>0$, for all $u > 0$;
\item [$(g_{0})$] $\displaystyle \lim_{u\to 0^{+}} \dfrac{g(u)}{u} = 0$ and $\displaystyle \lim_{\substack{u\to0^{+} \\ \omega\to1}}\dfrac{g(\omega u)}{g(u)}=1$;
\item [$(g_{\infty})$] $\displaystyle\lim_{\substack{u\to+\infty \\ \omega\to1}}\dfrac{g(\omega u)}{g(u)}=1$.
\end{itemize}
Then, there exists $\lambda^{*}>0$ such that for every $\lambda>\lambda^{*}$ there exist at least two positive radial solutions of problem \eqref{eq-main-pde}.
\end{theorem}

Notice that no Sobolev subcriticality assumptions are required for the nonlinear term: for instance, the function $g(u) = u^{p}$ enters the setting of Theorem~\ref{th-main} for every $p > 1$. The only restrictions on the growth at infinity come from assumption
$(g_{\infty})$: as an example, a function behaving, for $u$ large, like $e^u$ cannot be treated. Observe that a dual condition is required also at $u = 0$; for some comments about these assumptions (of regular-oscillation type) we refer to \cite[p.~452--453]{BoFeZa-16}.

\begin{figure}[!htb]
\begin{tikzpicture}[scale=1]
\begin{axis}[
  tick label style={font=\scriptsize},
  axis y line=left, 
  axis x line=middle,
  xtick={0.359781, 1.39176, 2.60244, 4.3119 ,5},
  ytick={-1,0,1},
  xticklabels={,,,,$5$},
  yticklabels={$-1$,$0$,$1$},
  xlabel={\small $|x|$},
  ylabel={\small $a(|x|)$},
every axis x label/.style={
    at={(ticklabel* cs:1.0)},
    anchor=west,
},
every axis y label/.style={
    at={(ticklabel* cs:1.0)},
    anchor=south,
},
  width=5.5cm,
  height=4.5cm,
  xmin=0,
  xmax=6,
  ymin=-1.5,
  ymax=1.5]
\addplot graphics[xmin=0,xmax=5,ymin=-1.1,ymax=1.1] {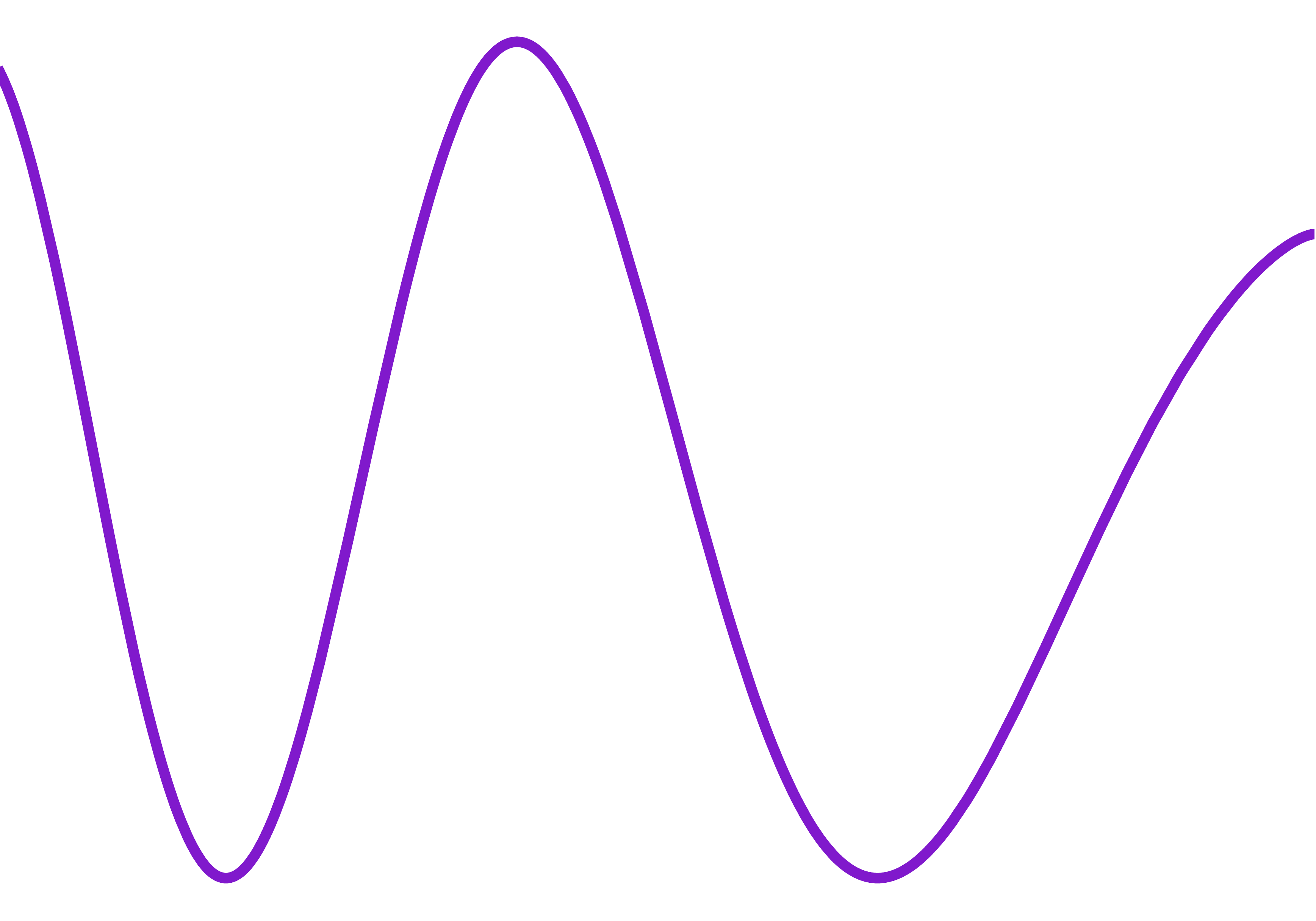};
\end{axis}
\end{tikzpicture}
\qquad
\begin{tikzpicture}[scale=1]
\begin{axis}[
  tick label style={font=\scriptsize,major tick length=3pt},
          scale only axis,
  enlargelimits=false,
  xtick={0, 0.359781, 1.39176, 2.60244, 4.3119 ,5},
  xticklabels={$0$,,,,,$5$},
  ytick={0,8},
  max space between ticks=50,
                minor y tick num=3,                
  xlabel={\small $|x|$},
  ylabel={\small $u(|x|)$},
every axis x label/.style={
below,
at={(1.9cm,0.1cm)},
  yshift=-8pt
  },
every axis y label/.style={
below,
at={(0cm,1.4cm)},
  xshift=-3pt},
  y label style={rotate=90,anchor=south},
  width=3.8cm,
  height=2.8cm,  
  xmin=0,
  xmax=5,
  ymin=0,
  ymax=8]
\addplot graphics[xmin=0,xmax=5,ymin=0,ymax=8] {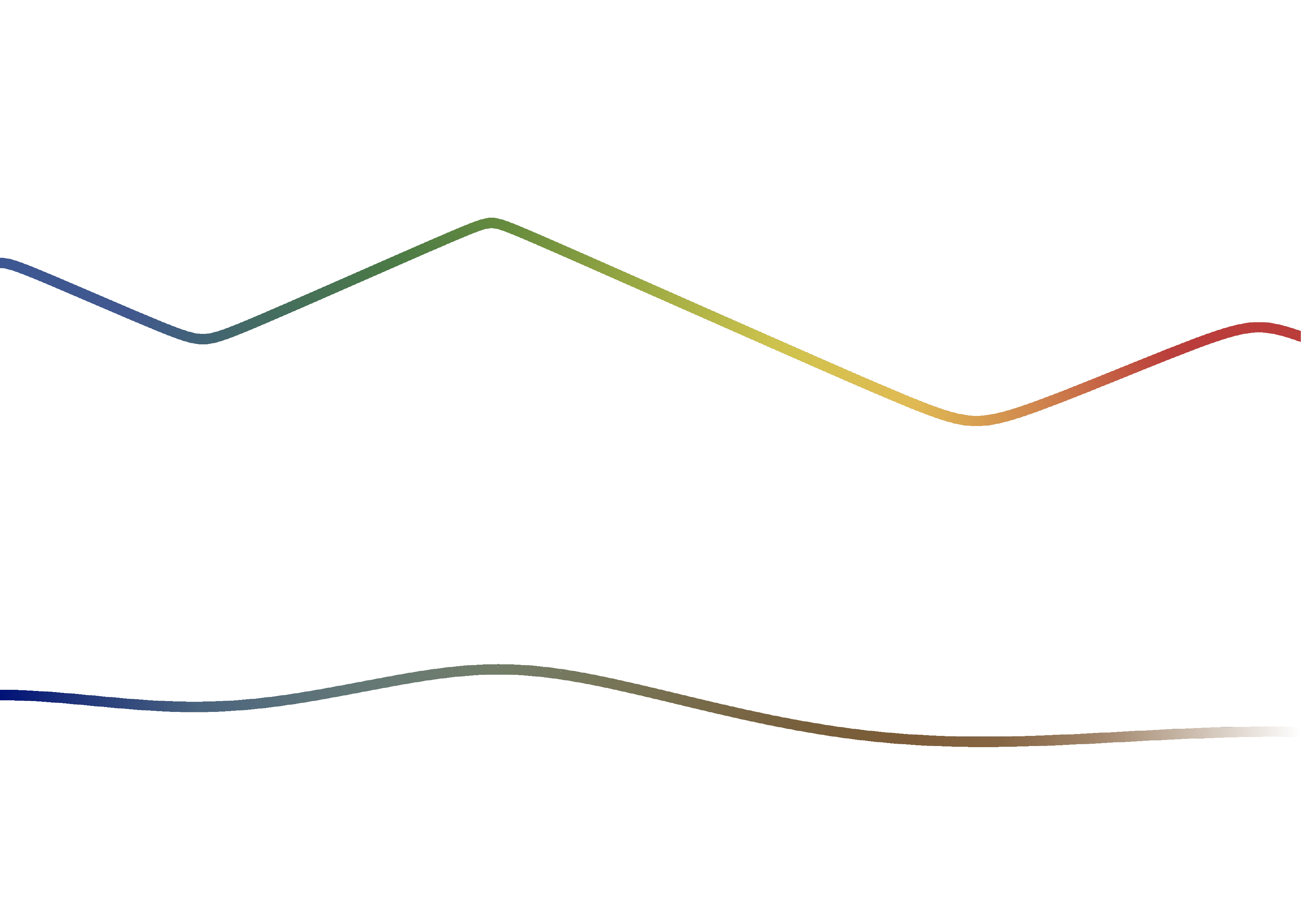};
\end{axis}
\end{tikzpicture}
\vspace{10pt}\\
\includegraphics[width=0.46\linewidth]{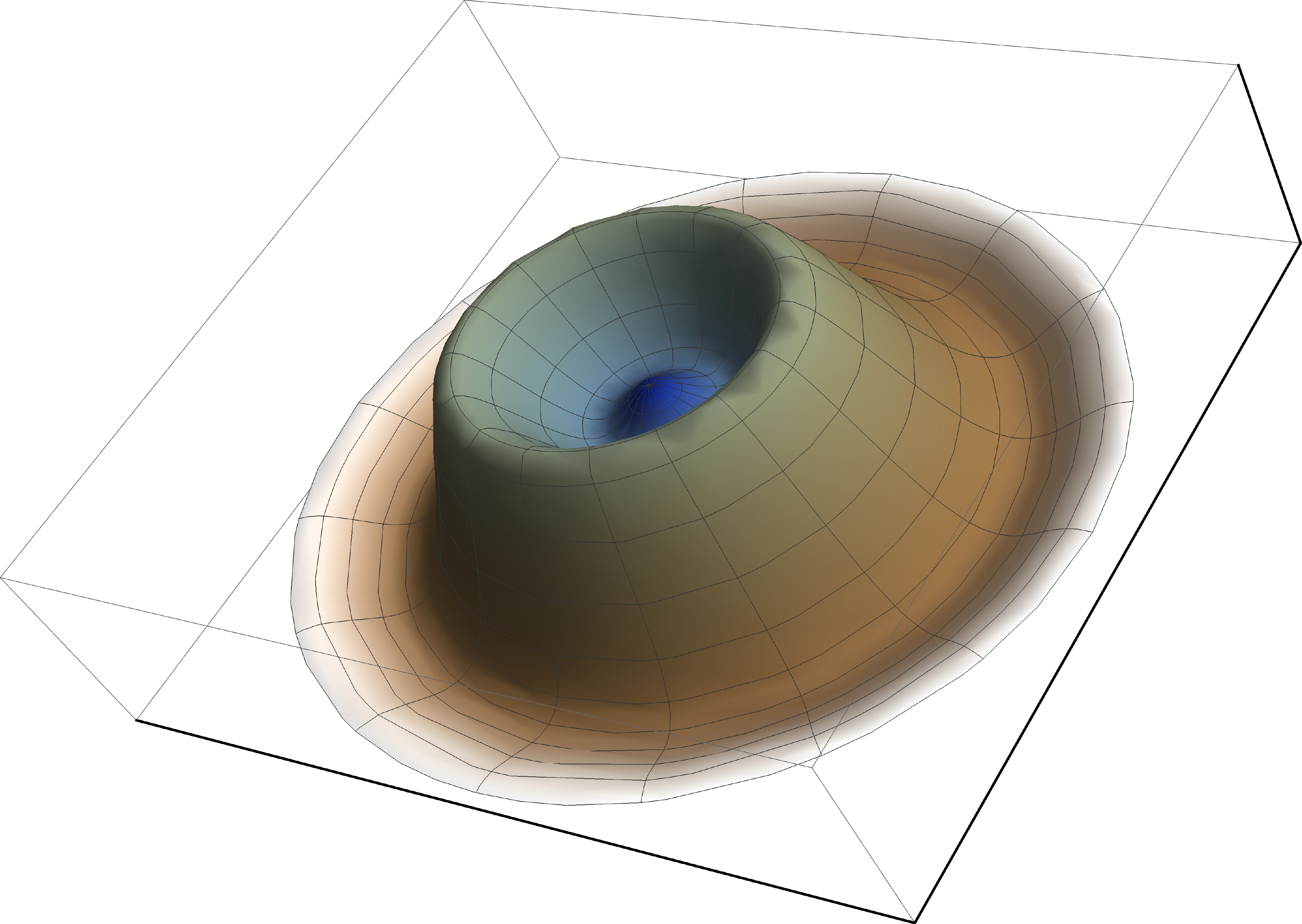}
\includegraphics[width=0.46\linewidth]{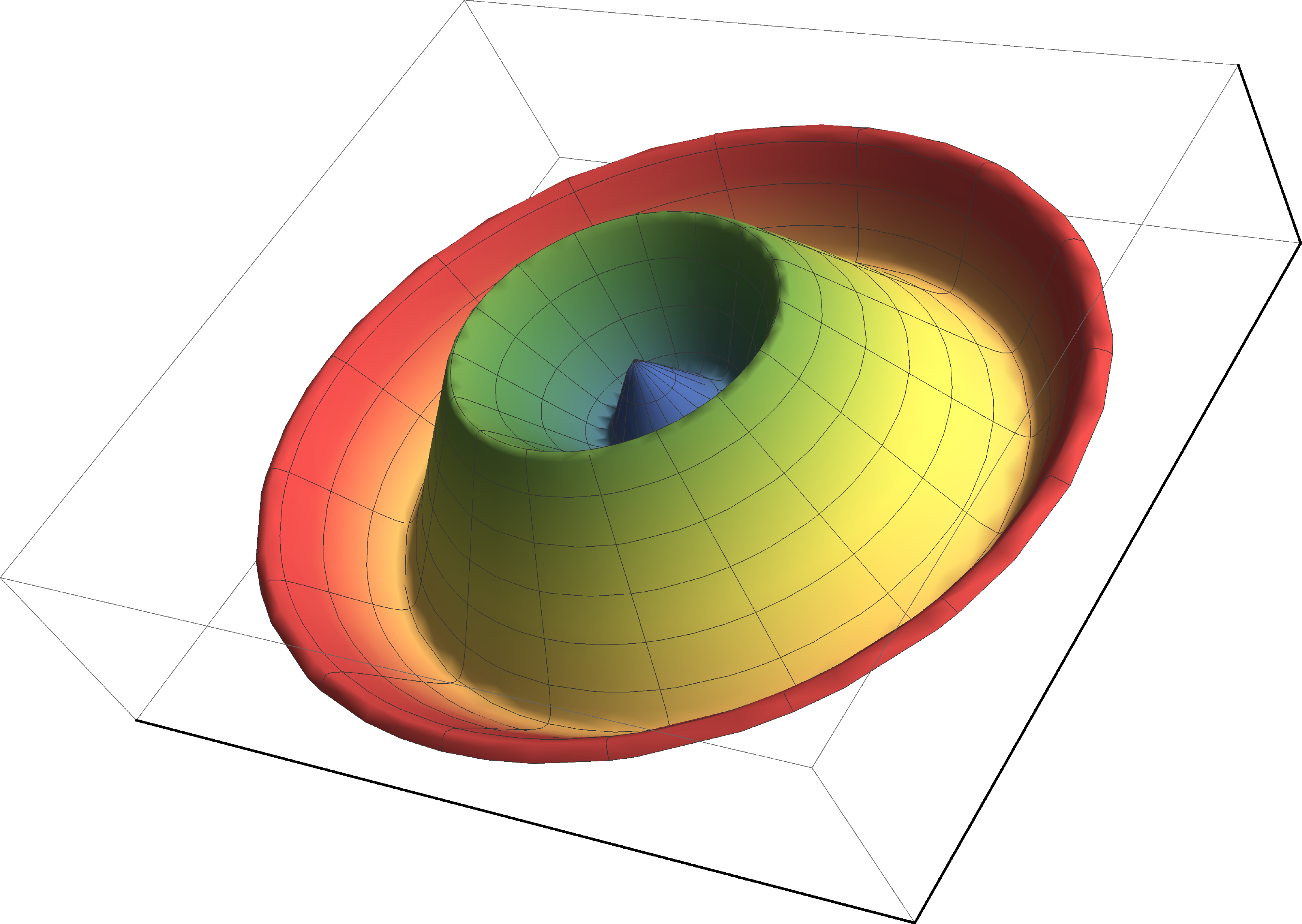}
\caption{The picture shows the graph of the weight function $a(|x|)=(\cos(||x|-5|^{3/2}+1)$ in $\mathopen{[}0,R\mathclose{]}=\mathopen{[}0,5\mathclose{]}$ and the graphs of the two radial solutions of problem \eqref{eq-main-pde} where $N=2$, $g(u)=u^{2}+u^{3}$ and $\lambda=0.1$. We notice that the large solution appears quite ``sharp-cornered'', in agreement with the analysis in \cite[Section~4]{BoFe-PP} and \cite[Section~3.2]{BoGa-19ccm}.}    
\label{fig-01}
\end{figure}

The proof of our main result is based again on topological degree theory, but, of course, with some differences with respect to the approach in \cite{BoFe-PP}. In particular, after having converted the radial Neumann problem \eqref{eq-main-pde} into the (singular) one-dimensional problem
\begin{equation}\label{eq-main}
\begin{cases}
\, \Biggl{(} r^{N-1}\dfrac{u'}{\sqrt{1-(u')^{2}}}\Biggr{)}' + \lambda r^{N-1} a(r) g(u) = 0,\\
\, u'(0) = u'(R) = 0,
\end{cases}
\end{equation}
we write it as a fixed point equation in the Banach space $\mathcal{C}(\mathopen{[}0,R\mathclose{]})$, by exploiting a strategy comparable to the one used, for instance, in \cite{BeJeMa-10,GaMaZa-93} (see also Remark~\ref{rem-2.1}). At this point, Leray--Schauder degree theory can be applied: the computation of the degree on three different balls of the space $\mathcal{C}(\mathopen{[}0,R\mathclose{]})$ leads to the result, similarly as in \cite{BoFe-PP,BoFeZa-16}.

The paper is organized as follows. In Section~\ref{section-2}, we describe the abstract setting and we present two lemmas for the computation of the degree. We give here all the details of the arguments, since it seems that no appropriate reference exists for the auxiliary results that we need. In Section~\ref{section-3}, we give the proof of Theorem~\ref{th-main}, by showing the applicability of the above mentioned lemmas, under the present assumptions. Compared to the case $N=1$, some technical difficulties arise when $N \geq 2$. First, the convexity (respectively, concavity) of the solutions is not a priori prescribed by the negativity (respectively, positivity) of the weight function $a(r)$. Second, 
when the weight function is positive near the center of the ball, as expected, some further care is needed in the estimates in order to exclude the possible appearance of a ``sudden loss of energy'' for the solutions (cf.~\cite{GaMaZa-97}). We stress that, due to the peculiar features of the Minkowski curvature operator, this can be successfully done even with no subcriticality assumptions on the nonlinear term.

\section{The abstract setting}\label{section-2}

In this section we present an abstract formulation of the problem and we prove some preliminary results based on the theory of the topological degree.

\subsection{The fixed point reduction}\label{section-2.1}

Throughout this section, we deal with the boundary value problem 
\begin{equation}\label{eq-phi}
\begin{cases}
\, \bigl{(} r^{N-1}\varphi(u')\bigr{)}' + r^{N-1} f(r,u) = 0, \\
\, u'(0) = u'(R) = 0,
\end{cases}
\end{equation} 
where $N \geq 1$ is an integer,
\begin{equation*}
\varphi(s) = \frac{s}{\sqrt{1-s^2}}, \quad s \in \mathopen{]}-1,1\mathclose{[},
\end{equation*}
and $f \colon \mathopen{[}0,R\mathclose{]} \times \mathbb{R} \to \mathbb{R}$ is an $L^{1}$-Carath\'{e}odory function (i.e.~$f(\cdot,u)$ is measurable for every $u \in \mathbb{R}$, $f(r,\cdot)$ is continuous for a.e. $r \in \mathopen{[}0,R\mathclose{]}$, and for every $K > 0$
there exists $\eta_k \in L^1(0,R)$ such that $|f(r,u)| \leq \eta_k(r)$ for a.e. $r \in \mathopen{[}0,R\mathclose{]}$ and for every 
$| u| \leq K$). Let us recall that a solution of \eqref{eq-phi} is a continuously differentiable function $u\colon \mathopen{[}0,R\mathclose{]} \to \mathbb{R}$, with $u'(0) = u'(R) = 0$ and $|u'(r)| < 1$ for every $r \in \mathopen{]}0,R\mathclose{[}$, such that the map $r \mapsto r^{N-1}\varphi(u')$ is absolutely continuous on $\mathopen{[}0,R\mathclose{]}$ and the differential equation in \eqref{eq-phi} is satisfied almost everywhere (see also Remark \ref{regolarita}).

It is convenient for the sequel to embed problem \eqref{eq-phi} into the two-parameter family of problems
\begin{equation}\label{eq-phi-theta}
\begin{cases}
\, \bigl{(} r^{N-1}\varphi(u')\bigr{)}' + \vartheta r^{N-1} \bigl{[} f(r,u) + \alpha v(r) \bigr{]} = 0, \\
\, u'(0) = u'(R) = 0,
\end{cases}
\end{equation} 
where $\vartheta\in\mathopen{[}0,1\mathclose{]}$, $\alpha\geq0$ and $v(r)$ is a fixed Lebesgue integrable function.
Notice that \eqref{eq-phi} is exactly \eqref{eq-phi-theta} with $\vartheta=1$ and $\alpha=0$.

Our aim is to write \eqref{eq-phi-theta} as a fixed point problem in the Banach space $\mathcal{C}(\mathopen{[}0,R\mathclose{]})$ of continuous functions in $\mathopen{[}0,R\mathclose{]}$, endowed with the usual norm $\|\cdot\|_{\infty}$. Accordingly, for $\vartheta\in\mathopen{[}0,1\mathclose{]}$ and $\alpha\geq0$, we introduce the operator
\begin{equation*}
T_{\vartheta,\alpha} \colon \mathcal{C}(\mathopen{[}0,R\mathclose{]}) \to \mathcal{C}(\mathopen{[}0,R\mathclose{]}),
\end{equation*}
defined as
\begin{equation}\label{operator}
\begin{aligned}
(T_{\vartheta,\alpha} u) (r) &= u(0) - \dfrac{1}{R^{N-1}} \int_{0}^{R} \zeta^{N-1} \bigl{[} f(\zeta,u(\zeta)) + \alpha v(\zeta) \bigr{]} \,\mathrm{d}\zeta
 \\ &\quad + \int_{0}^{r} \varphi^{-1} \biggl{(} -\dfrac{\vartheta}{\zeta^{N-1}} \int_{0}^{\zeta}\xi^{N-1} \bigl{[} f(\xi,u(\xi)) + \alpha v(\xi) \bigr{]}\,\mathrm{d}\xi \biggr{)} \,\mathrm{d}\zeta,
\end{aligned}
\end{equation}
for every $u\in\mathcal{C}(\mathopen{[}0,R\mathclose{]})$.
Notice that the definition of $T_{\vartheta,\alpha}$ is well-posed, since the function
\begin{equation*}
\mathopen{]}0,R\mathclose{]} \ni \zeta \mapsto \dfrac{1}{\zeta^{N-1}} \int_{0}^{\zeta}\xi^{N-1} \bigl{[} f(\xi,u(\xi)) + \alpha v(\xi) \bigr{]}\,\mathrm{d}\xi
\end{equation*}
continuously extends up to $\zeta=0$. For future convenience, we also observe that $T_{\vartheta,\alpha} u\in\mathcal{C}^{1}(\mathopen{[}0,R\mathclose{]})$ with
\begin{equation}\label{diff-operator}
(T_{\vartheta,\alpha} u)' (r) = 
\begin{cases}
\, \displaystyle \varphi^{-1} \biggl{(} -\dfrac{\vartheta}{r^{N-1}} \int_{0}^{r}\xi^{N-1} \bigl{[} f(\xi,u(\xi)) + \alpha v(\xi) \bigr{]}\,\mathrm{d}\xi \biggr{)}, &\text{if $r\in\mathopen{]}0,R\mathclose{]}$,} \\
\, 0, &\text{if $r=0$,}
\end{cases}
\end{equation}
for every $u\in\mathcal{C}(\mathopen{[}0,R\mathclose{]})$.

The following result holds true.

\begin{theorem}\label{th-operator}
Let $f \colon \mathopen{[}0,R\mathclose{]} \times \mathbb{R} \to \mathbb{R}$ be an $L^{1}$-Carath\'{e}odory function, $v \in L^1(0,R)$, 
$\vartheta\in\mathopen{[}0,1\mathclose{]}$, and $\alpha\geq0$. Then, the operator $T_{\vartheta,\alpha}$ defined in \eqref{operator} is completely continuous. Moreover, $u(r)$ is a solution of \eqref{eq-phi-theta} if and only if $u\in\mathcal{C}(\mathopen{[}0,R\mathclose{]})$ is a fixed point of $T_{\vartheta,\alpha}$.
\end{theorem}

\begin{proof}
The continuity of the operator $T_{\vartheta,\alpha}$ follows straightforwardly by observing that $T_{\vartheta,\alpha}$ is a composition of continuous maps. We claim that $T_{\vartheta,\alpha}$ sends bounded sets into  relatively compact sets. As a well-known consequence of Ascoli--Arzel\`{a} theorem, this is true if we prove that $\{T_{\vartheta,\alpha}u_{n}\}_{n}$ and $\{(T_{\vartheta,\alpha}u_{n})'\}_{n}$ are uniformly bounded for every bounded sequence $\{u_{n}\}_{n}$ in $\mathcal{C}(\mathopen{[}0,R\mathclose{]})$. This easily follows from the regularity assumptions on $f(r,u)$ and $v(r)$.

We prove the equivalence between the Neumann boundary value problem \eqref{eq-phi-theta} and the fixed point problem $u=T_{\vartheta,\alpha}u$.
Let $u(r)$ be a solution of \eqref{eq-phi-theta}, hence $u\in\mathcal{C}(\mathopen{[}0,R\mathclose{]})$. By integrating the equation in \eqref{eq-phi-theta} in $\mathopen{[}0,r\mathclose{]}$, recalling that $u'(0)=0$, and dividing by $r^{N-1}$, we obtain
\begin{equation}\label{eq-phi-th2.1}
\varphi(u'(r)) = -\dfrac{\vartheta}{r^{N-1}} \int_{0}^{r} \zeta^{N-1} \bigl{[} f(\zeta,u(\zeta)) + \alpha v(\zeta) \bigr{]} \,\mathrm{d}\zeta, \quad \text{for all $r\in\mathopen{]}0,R\mathclose{]}$.}
\end{equation}
Since $u'(R)=0$ we deduce that
\begin{equation}\label{eq-averR}
\dfrac{1}{R^{N-1}} \int_{0}^{R}\zeta^{N-1} f(\zeta,u(\zeta))\,\mathrm{d}\zeta =0.
\end{equation}
By applying $\varphi^{-1}$ to \eqref{eq-phi-th2.1} and integrating in $\mathopen{[}0,r\mathclose{]}$, we thus deduce that $u=T_{\vartheta,\alpha}u$.

On the other hand, let $u\in\mathcal{C}(\mathopen{[}0,R\mathclose{]})$ be such that $u=T_{\vartheta,\alpha}u$. Therefore, we have $u\in\mathcal{C}^{1}(\mathopen{[}0,R\mathclose{]})$ and $u'$ is given by the right-hand side of formula \eqref{diff-operator}. In particular, $u'(0)=0$ and $|u'(r)|<1$ for every $r\in\mathopen{]}0,R\mathclose{[}$. 
Next, by computing $u(0)=(Tu)(0)$, we obtain that \eqref{eq-averR} holds and so $u'(R)=0$.
By computing $\varphi(u')$, multiplying by $r^{N-1}$, and differentiating, we finally obtain that $u(r)$ solves \eqref{eq-phi-theta}. The proof is thus completed.
\end{proof}

\begin{remark}\label{regolarita}
Let us notice that, for every solution $u(r)$ of \eqref{eq-phi-theta}, from \eqref{eq-phi-th2.1} it follows that $\varphi(u')$ is absolutely continuous on $\mathopen{[}0,R\mathclose{]}$. Therefore, since $\varphi^{-1}$ is smooth, $u'$ is absolutely continuous as well.
In particular, if both $f(r,u)$ and $v(r)$ are continuous functions, then $u \in \mathcal{C}^2(\mathopen{[}0,R\mathclose{]})$ (cf.~\cite[Remark~3.3]{CoCoRi-14}).
$\hfill\lhd$
\end{remark}

As a consequence of Theorem~\ref{th-operator}, if $\Omega\subseteq \mathcal{C}(\mathopen{[}0,R\mathclose{]})$ is an open and bounded set such that
\begin{equation*}
u \neq T_{\vartheta,\alpha} u, \quad \text{for all $u\in\partial\Omega$,}
\end{equation*}
the Leray--Schauder degree $\mathrm{deg}_{\mathrm{LS}}(\mathrm{Id}-T_{\vartheta,\alpha},\Omega,0)$ is well-defined (we refer to \cite{De-85} for a classical reference about topological degree theory).

\begin{remark}\label{rem-2.1}
Dealing with the (one-dimensional) periodic boundary value problem
\begin{equation}\label{eq-per}
\begin{cases}
\, \bigl{(} \varphi(u')\bigr{)}' +  f(r,u) = 0, \\
\, u(0) = u(R), \; u'(0) = u'(R),
\end{cases}
\end{equation} 
a similar strategy can be followed in order to obtain a functional analytic formulation. Precisely, it can be seen (cf.~\cite{BeMa-07,MaMa-98}) that $u(r)$ is a solution of \eqref{eq-per} if and only if 
$u \in \mathcal{C}(\mathopen{[}0,R\mathclose{]})$ is a fixed point of the operator $\widetilde T \colon \mathcal{C}(\mathopen{[}0,R\mathclose{]}) \to \mathcal{C}(\mathopen{[}0,R\mathclose{]})$ defined as 
\begin{align*}
 (\widetilde T u) (r) &= u(0) -\int_{0}^{R} f(\zeta,u(\zeta))  \,\mathrm{d}\zeta
 \\ &\quad + \int_{0}^{r} \varphi^{-1} \biggl{(} - \int_{0}^{\zeta} f(\xi,u(\xi))\,\mathrm{d}\xi + Q_{\varphi}\left( \int_0^{\odot} 
f(\eta,u(\eta))\,\mathrm{d}\eta\right) \biggr{)} \,\mathrm{d}\zeta,
\end{align*}
where, for each $h \in \mathcal{C}(\mathopen{[}0,R\mathclose{]})$, $Q_\varphi(h) = Q_\varphi(h(\odot)) \in \mathbb{R}$ is defined as the unique solution of the equation
\begin{equation*}
\int_0^R \varphi^{-1}\left( -h(r) + Q_\varphi(h)\right)\,\mathrm{d}r = 0.
\end{equation*}
It is worth noticing that, compared with the Neumann case, this formulation is slightly less transparent: indeed, due to the non-locality of the
periodic boundary conditions, the additional term $Q_\varphi$ appears.  
An alternative fixed point formulation for \eqref{eq-per}, relying on a direct use of coincidence degree theory for the equivalent planar system
$u' = \varphi^{-1}(v)$, $v' = - f(r,u)$, has been recently proposed in \cite{FeZa-17tmna}.
$\hfill\lhd$
\end{remark}

\subsection{Two degree lemmas}\label{section-2.2}

Taking advantage of the abstract setting just presented, we now prove two lemmas for the computation of the degree on open balls $B(0,d)$ (with center $0$ and radius $d>0$) of the Banach space $\mathcal{C}(\mathopen{[}0,R\mathclose{]})$ in the framework of the Neumann problem
\begin{equation}\label{eq-phi-ag}
\begin{cases}
\, \bigl{(} r^{N-1}\varphi(u')\bigr{)}' + \lambda r^{N-1} a(r)g(u) = 0, \\
\, u'(0) = u'(R) = 0,
\end{cases}
\end{equation} 
where $a \colon \mathopen{[}0,R\mathclose{]} \to \mathbb{R}$ is an $L^1$-function, $g \colon \mathopen{[}0,+\infty\mathclose{[} \to \mathopen{[}0,+\infty\mathclose{[}$ is a continuous function satisfying $g(0) = 0$, and $\lambda>0$. 
Notice that, to enter the setting of the previous section, we need for the nonlinearity appearing in the equation to be defined for every $u \in \mathbb{R}$; accordingly, we set
\begin{equation*}
f(r,u) = \begin{cases}
\, \lambda a(r) g(u), & \text{if $u\geq0$,} \\
\, -u,& \text{if $u<0$.} 
\end{cases}
\end{equation*}
Observe that, due to the assumptions on $a(r)$ and $g(u)$, the function $f(r,u)$ is $L^{1}$-Carath\'{e}odory.

Recalling the definition \eqref{operator} of $T_{\vartheta,\alpha}$, we thus compute the Leray--Schauder degree of the map $\mathrm{Id}-T_{1,0}$. Observe that, by standard maximum principle arguments (based on the monotonicity of the map $r \mapsto r^{N-1}\varphi(u'(r))$ when $u(r)<0$), $u = T_{1,0} u$ implies that $u(r)$ is a \emph{non-negative} solution of \eqref{eq-phi} and thus solves \eqref{eq-phi-ag}.

The first lemma gives conditions for zero degree.

\begin{lemma}\label{lem-deg0}
Let $a \colon \mathopen{[}0,R\mathclose{]} \to \mathbb{R}$ be an $L^1$-function, let $g \colon \mathopen{[}0,+\infty\mathclose{[} \to \mathopen{[}0,+\infty\mathclose{[}$ be a continuous function satisfying $g(0) = 0$, and $\lambda>0$. 
Let $d>0$ and assume that there exists a non-negative function $v \in L^1(0,R)$, with $v\not\equiv0$, such that the following properties hold:
\begin{itemize}
\item[$(H_{1})$]
If $\alpha \geq 0$ and $u(r)$ is a non-negative solution of
\begin{equation}\label{eq-lem-deg0}
\begin{cases}
\, \bigl{(} r^{N-1}\varphi(u')\bigr{)}' + \lambda r^{N-1} a(r) g(u) + \alpha r^{N-1} v(r) = 0, \\
\, u'(0)=u'(R)=0,
\end{cases}
\end{equation}
then $\|u\|_{\infty}\neq d$.
\item[$(H_{2})$]
There exists $\alpha_{0} \geq 0$ such that problem \eqref{eq-lem-deg0}, with $\alpha=\alpha_{0}$, has no non-negative solutions $u(r)$ with $\|u\|_{\infty}\leq d$.
\end{itemize}
Then, it holds that $\mathrm{deg}_{\mathrm{LS}}(\mathrm{Id}-T_{1,0},B(0,d),0) = 0$.
\end{lemma}

\begin{proof}
First, recalling Theorem~\ref{th-operator}, we have that $u(r)$ is a solution of 
\begin{equation}\label{eq-lem-deg0f}
\begin{cases}
\, \bigl{(} r^{N-1}\varphi(u')\bigr{)}' + r^{N-1} \bigl{[} f(r,u) + \alpha v(r) \bigr{]}= 0, \\
\, u'(0)=u'(R)=0,
\end{cases}
\end{equation}
if and only if $u = T_{1,\alpha} u$.
By maximum principle arguments, since $v \geq 0$, every solution of \eqref{eq-lem-deg0f} is non-negative and thus solves \eqref{eq-lem-deg0}.
Therefore, for every $\alpha \geq 0$, condition $(H_{1})$ ensures that 
\begin{equation*}
u \neq T_{1,\alpha} u, \quad \text{for all $u\in \partial B(0,d)$.}
\end{equation*}
We deduce that $\mathrm{deg}_{\mathrm{LS}}(\mathrm{Id}-T_{1,\alpha},B(0,d),0)$ is well-defined for every $\alpha \geq 0$. As a final step, we apply the homotopy invariance property of the degree to conclude that
\begin{equation*}
\mathrm{deg}_{\mathrm{LS}}(\mathrm{Id}-T_{1,0},B(0,d),0)=\mathrm{deg}_{\mathrm{LS}}(\mathrm{Id}-T_{1,\alpha_{0}},B(0,d),0)=0,
\end{equation*}
where the last equality follows from hypothesis $(H_{2})$.
\end{proof}

The second lemma ensures non-zero degree. Notice that here we need to assume further conditions on $a(r)$ and $g(u)$.

\begin{lemma}\label{lem-deg1}
Let $a \colon \mathopen{[}0,R\mathclose{]} \to \mathbb{R}$ be an $L^1$-function satisfying $(a_{\#})$, let $g \colon \mathopen{[}0,+\infty\mathclose{[} \to \mathopen{[}0,+\infty\mathclose{[}$ be a continuous function satisfying $(g_{*})$, and $\lambda>0$. Let $d>0$ and assume that the following property holds:
\begin{itemize}
\item[$(H_{3})$]
If $\vartheta\in \mathopen{]}0,1\mathclose{]}$ and $u(r)$ is a non-negative solution of
\begin{equation}\label{eq-lem-deg1}
\begin{cases}
\, \bigl{(} r^{N-1}\varphi(u')\bigr{)}' + \vartheta \lambda r^{N-1} a(r) g(u) = 0, \\
\, u'(0)=u'(R)=0,
\end{cases}
\end{equation}
then $\|u\|_{\infty} \neq d$.
\end{itemize}
Then, it holds that $\mathrm{deg}_{\mathrm{LS}}(\mathrm{Id}-T_{1,0},B(0,d),0) = 1$.
\end{lemma}

\begin{proof}
First, recalling Theorem~\ref{th-operator}, we have that $u(r)$ is a solution of
\begin{equation}\label{eq-lem-deg11f}
\begin{cases}
\, \bigl{(} r^{N-1}\varphi(u')\bigr{)}' + \vartheta r^{N-1} f(r,u) = 0, \\
\, u'(0)=u'(R)=0,
\end{cases}
\end{equation}
if and only if $u = T_{\vartheta,0} u$.
By maximum principle arguments, every solution of \eqref{eq-lem-deg11f} is non-negative and thus solves \eqref{eq-lem-deg1}.

Hypothesis $(H_{3})$ implies that $\mathrm{deg}_{\mathrm{LS}}(\mathrm{Id}-T_{\vartheta,0},B(0,d),0)$ is well-defined for every $\vartheta\in\mathopen{]}0,1\mathclose{]}$.
Let us consider the case $\vartheta=0$. The fixed point problem $u = T_{0,0} u$ reduces to
\begin{equation*}
u(r) = (T_{0,0} u)(r) = u(0) - \dfrac{1}{R^{N-1}} \int_{0}^{R} \zeta^{N-1} f(\zeta,u(\zeta)) \,\mathrm{d}\zeta, \quad u\in \mathcal{C}(\mathopen{[}0,R\mathclose{]}),
\end{equation*}
whose solutions are constant functions.
Observing that
\begin{equation*}
(\mathrm{Id}-T_{0,0}) u \equiv  - \dfrac{1}{R^{N-1}} \int_{0}^{R} \zeta^{N-1} f(\zeta,s) \,\mathrm{d}\zeta, \quad \text{for all $u\in \mathcal{C}(\mathopen{[}0,R\mathclose{]})$ with $u\equiv s\in\mathbb{R}$,}
\end{equation*}
we consider the function
\begin{equation*}
f^{\#}(s) = \dfrac{1}{R^{N-1}} \int_{0}^{R} \zeta^{N-1} f(\zeta,s) \,\mathrm{d}\zeta =
\begin{cases}
\, -\dfrac{R}{N} s, & \text{if $s \leq 0$,} \\
\, \dfrac{\lambda}{R^{N-1}} \biggl{(}\displaystyle \int_{0}^{R} \zeta^{N-1}a(\zeta) \,\mathrm{d}\zeta \biggr{)} g(s), & \text{if $s \geq 0$.}
\end{cases}
\end{equation*}
By hypothesis $(a_{\#})$ we deduce that
\begin{equation*}
\int_{0}^{R} r^{N-1} a(r) \,\mathrm{d}r = \dfrac{1}{\omega_{N}} \int_{B}a(|x|) \,\mathrm{d}x < 0
\end{equation*}
(where $\omega_{N}$ is the measure of the unit sphere in $\mathbb{R}^{N}$) 
and, consequently, we obtain that $f^{\#}(s)s < 0$ for $s \neq 0$.
As a consequence of the excision property, we can reduce the study of the degree of $\mathrm{Id}-T_{0,0}$ in the set $B(0,d)\cap\mathbb{R}=\mathopen{]}-d,d\mathclose{[}$.
Since $f^{\#}(s)$ has no zeros on $\partial(B(0,d)\cap\mathbb{R})=\{\pm d\}$ and, more precisely, $f^{\#}(d)<0<f^{\#}(-d)$, we deduce that
\begin{equation*}
\mathrm{deg}_{\mathrm{LS}}(\mathrm{Id}-T_{0,0},B(0,d),0) = \mathrm{deg}_{\mathrm{B}}(-f^{\#},\mathopen{]}-d,d\mathclose{[},0) = 1.
\end{equation*}
Therefore, by the homotopy invariance property of the degree we infer that
\begin{equation*}
\mathrm{deg}_{\mathrm{LS}}(\mathrm{Id}-T_{1,0},B(0,d),0) = \mathrm{deg}_{\mathrm{LS}}(\mathrm{Id}-T_{0,0},B(0,d),0) = 1.
\end{equation*}
This concludes the proof.
\end{proof}

\section{Proof of Theorem~\ref{th-main}}\label{section-3}

We look for radial solutions of problem \eqref{eq-main-pde}, meant as solutions of the one-dimensional boundary value problem \eqref{eq-main}. To investigate this, we rely on the abstract setting presented in Section~\ref{section-2}.

We divide the arguments of the proofs into some steps. 

\subsubsection*{Step~1. Fixing the constants $\delta^{*}$ and $\lambda^{*}$.}
We first set $I^{+}_{i}=\mathopen{[}\sigma_{i},\tau_{i}\mathclose{]}$ for $i=1,\ldots,m$, and we assume that these intervals are ordered in the natural way, that is
\begin{equation*}
0 \leq \sigma_{1} < \tau_{1} < \ldots < \sigma_m < \tau_m \leq R.
\end{equation*}
Notice that, whenever $a(r)$ vanishes on one or more non-degenerate intervals, the choice of some of the values $\sigma_{i}$ and $\tau_{i}$ is not unique.
Let $\varepsilon>0$ be such that
\begin{equation*}
\varepsilon<\dfrac{ |I^{+}_{i} |}{4} \quad \text{ and } \quad \int_{\sigma_{i}+2\varepsilon}^{\tau_{i}-2\varepsilon}  r^{N-1} a(r)\,\mathrm{d}r > 0, \quad \text{for every $i = 1,\ldots,m$,}
\end{equation*}
and let us define
\begin{equation*}
\delta^{*} = \dfrac{2^{N-1}\varepsilon^{N}}{R^{N-1}}
\end{equation*}
and
\begin{equation*}
\delta_{*} = \min_{i=1,\ldots,m} 
\dfrac{\delta^{*}}{1+2\varphi\biggl{(}\dfrac{1}{2} \biggr{)} |I^{+}_{i}| \dfrac{\tau_{i}^{2N-2}}{\sigma_{i}^{N-1}(2\varepsilon)^{N}}},
\end{equation*}
if $\sigma_{1}\neq0$, or 
\begin{equation*}
\delta_{*} = \min \left\{ 
\dfrac{\delta^{*}-\gamma}{1+2\varphi\biggl{(}\dfrac{1}{2} \biggr{)} |I^{+}_{1}| \dfrac{\tau_{1}^{2N-2}}{\gamma^{N-1}(2\varepsilon)^{N}}},
\min_{i=2,\ldots,m} \dfrac{\delta^{*}}{1+2\varphi\biggl{(}\dfrac{1}{2} \biggr{)} |I^{+}_{i}| \dfrac{\tau_{i}^{2N-2}}{\sigma_{i}^{N-1}(2\varepsilon)^{N}}} \right\}.
\end{equation*}
if $\sigma_{1}=0$, where $\gamma=\min\{\delta^{*},\tau_{1}\}/2$. Notice that, in both cases, $\delta_{*}\in\mathopen{]}0,\delta^{*}\mathclose{[}$.
Moreover, let
\begin{equation*}
\lambda^{*} = \max_{i=1,\ldots,m}\dfrac{2 R^{N-1}\varphi(1/2)}{ \min \bigl{\{} g(u) \colon u \in \mathopen{[} \delta_{*},\delta^{*} \mathclose{]} \bigr{\}} \displaystyle{\int_{\sigma_{i}+2\varepsilon}^{\tau_{i}-2\varepsilon} r^{N-1} a(r)\,\mathrm{d}r}}.
\end{equation*}
From now on, let $\lambda > \lambda^{*}$ be fixed.

\subsubsection*{Step~2. Computation of the degree in $B(0,\delta^{*})$.}
We are going to apply Lemma~\ref{lem-deg0} to the open ball $B(0,\delta^{*})$ taking as $v(r)$ the indicator function of the set $\bigcup_{i} I^{+}_{i}$.

First we verify condition $(H_{1})$. We suppose by contradiction that there exist $\alpha \geq 0$ and a non-negative solution $u(r)$ to \eqref{eq-lem-deg0} such that $\|u\|_{\infty} = \delta^{*}$. 

\smallskip
\noindent
\textit{Claim~1.} There exists $i=\{1,\ldots,m\}$ such that
\begin{equation}\label{eq-cl1}
\max_{r\in I^{+}_{i}} u(r) = \delta^{*}.
\end{equation}
Since $v\equiv0$ on $\mathopen{[}0,R\mathclose{]}\setminus \bigcup_{i} I^{+}_{i}$, by conditions $(a_{*})$ and $(g_{*})$ we deduce that the map $r\to r^{N-1}\varphi(u'(r))$ is non-increasing on each interval $I^{+}_{i}$ and non-decreasing on each interval $J \subseteq \mathopen{[}0,R\mathclose{]}\setminus \bigcup_{i} I^{+}_{i}$. 
We show that
\begin{equation}\label{eq-convexity}
\max_{r\in J}u(r) = \max_{r\in\partial J} u(r),
\end{equation}
for every interval $J \subseteq \mathopen{[}0,R\mathclose{]}\setminus \bigcup_{i} I^{+}_{i}$.
Indeed, let $J=\mathopen{[}\tau,\sigma\mathclose{]}$ and $\hat{r}\in\mathopen{]}\tau,\sigma\mathclose{[}$. If $u'(\hat{r}) \geq 0$, then $u'(r) \geq 0$ for all $r\in\mathopen{[}\hat{r},\sigma\mathclose{]}$, and so $u(\hat{r})\leq u(\sigma)$. Analogously, if $u'(\hat{r}) \leq 0$, then $u'(r) \leq 0$ for all $r\in\mathopen{[}\tau,\hat{r}\mathclose{]}$, and so $u(\hat{r})\leq u(\tau)$. Therefore, \eqref{eq-convexity} holds. 

We further observe that if $\tau=0$ then $u'(\tau)=0$ and so $\max_{r\in J}u(r) = u(\sigma)$; if $\sigma=R$ then $u'(\sigma)=0$ and so $\max_{r\in J}u(t) = u(\tau)$. As a consequence of this and \eqref{eq-convexity}, \eqref{eq-cl1} follows.

From now on we focus on the behavior of $u(r)$ on $I^{+}_{i}=\mathopen{[}\sigma_{i},\tau_{i}\mathclose{]}$. 

\smallskip
\noindent
\textit{Claim~2.} It holds that
\begin{equation}\label{estim-u'}
|u'(r)| \leq \dfrac{\tau_{i}^{N-1}}{(2\varepsilon)^{N}} \, u(r), \quad \text{for all $r\in \mathopen{[}\sigma_{i}+2\varepsilon,\tau_{i}-2\varepsilon\mathclose{]}$.}
\end{equation}
Indeed, let us fix $r\in \mathopen{[}\sigma_{i}+2\varepsilon,\tau_{i}-2\varepsilon\mathclose{]}$. If $u'(r)=0$ then the estimate is obvious. If $u'(r)>0$, by using the monotonicity of the map $r\to r^{N-1}\varphi(u'(r))$, we have that 
\begin{equation*}
u'(\xi) \geq \varphi^{-1} \biggl{(} \biggl{(}\dfrac{r}{\xi}\biggr{)}^{\!N-1} \varphi(u'(r)) \biggr{)} \geq \varphi^{-1} \bigl{(} \varphi(u'(r)) \bigr{)} = u'(r), \quad \text{for all $\xi\in\mathopen{]}\sigma_{i},r\mathclose{]}$.}
\end{equation*}
By integrating the above inequality in $\mathopen{[}\sigma_{i},r\mathclose{]}$ we obtain 
\begin{equation*}
u(r)\geq u(r)-u(\sigma_{i})= \int_{\sigma_{i}}^{r} u'(\xi)\,\mathrm{d}\xi
\geq (r-\sigma_{i}) u'(r)\geq 2\varepsilon u'(r) \geq \dfrac{(2\varepsilon)^{N}}{\tau_{i}^{N-1}} u'(r),
\end{equation*}
where the last inequality follows from $2\varepsilon<\tau_{i}$. This implies~\eqref{estim-u'}. As last case, if $u'(r)<0$, by arguing as above, we have that 
\begin{align*}
-u'(\xi) &\geq \varphi^{-1} \biggl{(} \biggl{(}\dfrac{r}{\xi}\biggr{)}^{\!N-1} \varphi(-u'(r)) \biggr{)} 
\geq -\biggl{(}\dfrac{r}{\xi}\biggr{)}^{\!N-1}  u'(r)
\\
& \geq -\biggl{(}\dfrac{r}{\tau_{i}}\biggr{)}^{\!N-1}  u'(r), \quad \text{for all $\xi\in\mathopen{[}r,\tau_{i}\mathclose{]}$,}
\end{align*}
where we have used the oddness of $\varphi$ and $\varphi^{-1}$, and the elementary inequality
\begin{equation*}
\varphi^{-1} \bigl{(} \vartheta \varphi(s) \bigr{)} \geq \vartheta s, \quad \text{for all $\vartheta\in\mathopen{[}0,1\mathclose{]}$ and $s\in\mathopen{[}0,1\mathclose{[}$,}
\end{equation*}
coming from the convexity of $\varphi$ in $\mathopen{[}0,1\mathclose{[}$.
Then, by integrating in $\mathopen{[}r,\tau_{i}\mathclose{]}$ we obtain 
\begin{equation*}
u(r)\geq u(r)-u(\tau_{i}) = - \int_{r}^{\tau_{i}} u'(\xi)\,\mathrm{d}\xi 
\geq -(\tau_{i}-r) \biggl{(}\dfrac{r}{\tau_{i}}\biggr{)}^{\!N-1}u'(r) 
\geq - \dfrac{(2\varepsilon)^{N}}{\tau_{i}^{N-1}} u'(r),
\end{equation*}
finally implying \eqref{estim-u'}.

For further convenience, we observe that from \eqref{estim-u'} we have in particular that
\begin{equation}\label{u'_1_2}
|u'(r)| \leq \dfrac{1}{2}, \quad \text{for all $r \in \mathopen{[}\sigma_{i}+2\varepsilon,\tau_{i}-2\varepsilon\mathclose{]}$.}
\end{equation}

\smallskip
\noindent
\textit{Claim~3.} It holds that
\begin{equation}\label{eq-delta}
\min_{r\in \mathopen{[}\sigma_{i}+2\varepsilon,\tau_{i}-2\varepsilon\mathclose{]}} u(r)\geq \delta_{*}.
\end{equation}
To show this, let $r^{*}\in I^{+}_{i}$ and $\check{r}\in\mathopen{[}\sigma_{i}+2\varepsilon,\tau_{i}-2\varepsilon\mathclose{]}$ be such that
\begin{equation*}
u(r^{*}) = \max_{r\in \mathopen{[}\sigma_{i},\tau_{i}\mathclose{]}} u(r) = \delta^{*}, \quad u(\check{r}) = \min_{r\in \mathopen{[}\sigma_{i}+2\varepsilon,\tau_{i}-2\varepsilon\mathclose{]}} u(r).
\end{equation*}
The case $r^{*}=\check{r}$ is trivial. So we consider the two cases: $r^{*}<\check{r}$ and $r^{*}>\check{r}$. 

If $\sigma_{i} \leq r^{*} < \check{r}\in \mathopen{[}\sigma_{i}+2\varepsilon,\tau_{i}-2\varepsilon\mathclose{]}$, then $u'(r^{*}) \leq 0$. 
Using the monotonicity of $r\mapsto r^{N-1} \varphi(u'(r))$ in $\mathopen{[}\sigma_{i},\tau_{i}\mathclose{]}$, we deduce that
\begin{equation}\label{eq-monot1}
u'(\xi) \geq \varphi^{-1} \biggl{(} \biggl{(} \dfrac{\zeta}{\xi}\biggr{)}^{\! N-1} \varphi(u'(\zeta)) \biggr{)}, \quad \text{for all $\xi,\zeta\in \mathopen{]}\sigma_{i},\tau_{i}\mathclose{]}$ with $\xi\leq\zeta$,}
\end{equation}
and so, from \eqref{eq-monot1} with $\xi=r^{*}$, $u'(\zeta) \leq 0$ for all $\zeta\in \mathopen{[}r^{*},\check{r}\mathclose{]}$, in particular $u'(\check{r}) \leq 0$. 
Assume now $i\geq 2$, or $i =1$ and $\sigma_{1}\neq 0$. 
An integration of \eqref{eq-monot1} with $\zeta=\check{r}$ on $\mathopen{[}r^{*},\check{r}\mathclose{]}$ leads to
\begin{align*}
\delta^{*} - u(\check{r}) &= - \int_{r^{*}}^{\check{r}} u'(\xi) \,\mathrm{d}\xi 
\leq \int_{r^{*}}^{\check{r}} \varphi^{-1} \biggl{(} \biggl{(} \dfrac{\check{r}}{\xi}\biggr{)}^{\! N-1} \varphi(-u'(\check{r})) \biggr{)} \,\mathrm{d}\xi
\\ &\leq \int_{r^{*}}^{\check{r}} \biggl{(} \dfrac{\check{r}}{\xi}\biggr{)}^{\! N-1} \varphi(-u'(\check{r})) \,\mathrm{d}\xi
\leq (\check{r}-r^{*}) \biggl{(} \dfrac{\check{r}}{r^{*}}\biggr{)}^{\! N-1} \varphi(-u'(\check{r}))
\\ &\leq |I^{+}_{i}| \biggl{(} \dfrac{\tau_{i}}{\sigma_{i}}\biggr{)}^{\! N-1} \varphi(-u'(\check{r})).
\end{align*}
Using \eqref{u'_1_2} together with the fact that
\begin{equation*}
\varphi(s) \leq 2\varphi\biggl{(}\dfrac{1}{2} \biggr{)}s, \quad \text{for all $s\in\biggl{[}0,\dfrac{1}{2}\biggr{]}$,}
\end{equation*}
and estimate \eqref{estim-u'}, we obtain
\begin{equation*}
\delta^{*} - u(\check{r}) \leq 2\varphi\biggl{(}\dfrac{1}{2} \biggr{)} |I^{+}_{i}| \dfrac{\tau_{i}^{2N-2}}{\sigma_{i}^{N-1}(2\varepsilon)^{N}} \, u(\check{r}).
\end{equation*}
Then, \eqref{eq-delta} holds.
Consider now the case $i=1$ and $\sigma_{1}=0$. Then, recalling that $u'(\xi)\leq 0$ on $\mathopen{[}0,\tau_{1}\mathclose{]}$, we have $r^{*}=0$ and $\check{r}=\tau_{1}-2 \varepsilon$. Now, arguing as above and using the fact that $|u'(\xi)|<1$ on $\mathopen{[}0,\tau_{1}\mathclose{]}$, we deduce
\begin{align*}
\delta^{*} - u(\tau_{1}-2\varepsilon) &= - \int_{0}^{\tau_{1}-2 \varepsilon} u'(\xi) \,\mathrm{d}\xi 
= - \int_{0}^{\gamma} u'(\xi) \,\mathrm{d}\xi - \int_{\gamma}^{\tau_{1}-2\varepsilon} u'(\xi) \,\mathrm{d}\xi
\\&\leq \gamma + \int_{\gamma}^{\tau_{1}-2\varepsilon} \varphi^{-1} \biggl{(} \biggl{(} \dfrac{\tau_{1}-2\varepsilon}{\xi}\biggr{)}^{\! N-1} \varphi(-u'(\tau_{1}-2\varepsilon)) \biggr{)} \,\mathrm{d}\xi
\\ &\leq \gamma + 2\varphi\biggl{(}\dfrac{1}{2} \biggr{)} |I^{+}_{1}| \dfrac{\tau_{1}^{2N-2}}{\gamma^{N-1}(2\varepsilon)^{N}} \, u(\tau_{1}-2 \varepsilon).
\end{align*}
Then, \eqref{eq-delta} holds.

On the other hand, if $\tau_{i}\geq r^{*} > \check{r}\in \mathopen{[}\sigma_{i}+2\varepsilon,\tau_{i}-2\varepsilon\mathclose{]}$, then $u'(r^{*}) \geq 0$. Notice that this can happen only when $i\geq 2$, or $i=1$ and $\sigma_{1}\neq 0$.
Using the monotonicity of $r\mapsto r^{N-1} \varphi(u'(r))$ in $\mathopen{[}\sigma_{i},\tau_{i}\mathclose{]}$, we deduce that
\begin{equation}\label{eq-monot2}
\varphi^{-1} \biggl{(} \biggl{(} \dfrac{\xi}{\zeta}\biggr{)}^{\! N-1} \varphi(u'(\xi)) \biggr{)} \geq u'(\zeta), \quad \text{for all $\xi,\zeta\in \mathopen{]}\sigma_{i},\tau_{i}\mathclose{]}$ with $\xi\leq\zeta$,}
\end{equation}
and so, from \eqref{eq-monot2} with $\zeta=r^{*}$, $u'(\xi) \geq 0$ for all $\xi\in \mathopen{[}\check{r},r^{*}\mathclose{]}$, in particular $u'(\check{r}) \geq 0$. 
An integration of \eqref{eq-monot2} with $\xi=\check{r}$ on $\mathopen{[}\check{r},r^{*}\mathclose{]}$ and an application of \eqref{estim-u'} lead to
\begin{align*}
\delta^{*} - u(\check{r}) &= \int_{\check{r}}^{r^{*}} u'(\zeta) \,\mathrm{d}\zeta
\leq \int_{\check{r}}^{r^{*}} \varphi^{-1} \biggl{(} \biggl{(} \dfrac{\check{r}}{\zeta}\biggr{)}^{\! N-1} \varphi(u'(\check{r})) \biggr{)} \,\mathrm{d}\zeta
\\ &\leq \int_{\check{r}}^{r^{*}}\varphi^{-1} \bigl{(} \varphi(u'(\check{r})) \bigr{)} \,\mathrm{d}\zeta
\leq |I^{+}_{i}| \dfrac{\tau_{i}^{N-1}}{(2\varepsilon)^{N}} \, u(\check{r})
\leq |I^{+}_{i}| \dfrac{\tau_{i}^{N-1}}{(2\varepsilon)^{N}} \dfrac{\tau_{i}^{N-1}}{\sigma_{i}^{N-1}} \, u(\check{r}).
\end{align*}
Then, we have \eqref{eq-delta}.

\smallskip

We are now ready to verify condition $(H_{1})$ of Lemma~\ref{lem-deg0}. We integrate the equation in \eqref{eq-lem-deg0} on $\mathopen{[}\sigma_{i}+2\varepsilon,\tau_{i}-2\varepsilon\mathclose{]}$ and, recalling \eqref{u'_1_2}, \eqref{eq-delta} and that $\varphi$ is odd, we obtain
\begin{equation*}
\lambda \min \bigl{\{} g(u) \colon u \in \mathopen{[}\delta_{*},\delta^{*} \mathclose{]}\bigr{\}} \int_{\sigma_{i}+2\varepsilon}^{\tau_{i}-2\varepsilon} r^{N-1}a(r)\,\mathrm{d}r \leq 2 R^{N-1}\varphi \biggl{(} \dfrac{1}{2}\biggr{)}, 
\end{equation*}
a contradiction with respect to $\lambda > \lambda^{*}$.

As for assumption $(H_{2})$, we integrate the equation in \eqref{eq-lem-deg0} in $\mathopen{[}0,R\mathclose{]}$ and, passing to the absolute value, we deduce
\begin{equation*}
\alpha \| r^{N-1}v \|_{L^{1}} \leq \lambda R^{N-1} \|a\|_{L^{1}} \max_{u \in \mathopen{[}0,\delta^{*}\mathclose{]}} g(u).
\end{equation*}
A contradiction follows for $\alpha$ sufficiently large. From Lemma~\ref{lem-deg0} we thus obtain
\begin{equation}\label{deg-B0rho*}
\mathrm{deg}_{\mathrm{LS}}(\mathrm{Id}-T_{1,0},B(0,\delta^{*})) = 0.
\end{equation}

\subsubsection*{Step~3. Fixing the constant $d^{*}$ and computation of the degree in $B(0,d^{*})$.}
First, we claim that there exists $d^{*}\in\mathopen{]}0,\delta^{*}\mathclose{[}$ such that, for every $\vartheta\in \mathopen{]}0,1\mathclose{]}$, every non-negative  solution $u(r)$ of \eqref{eq-lem-deg1} with $\|u\|_{\infty} \leq d^{*}$ is such that $u\equiv0$.

We assume, by contradiction, that there exists a sequence $\{u_{n}\}_{n}$ of non-negative solutions of \eqref{eq-lem-deg1} for $\vartheta=\vartheta_{n}$ satisfying $0<\|u_{n}\|_{\infty}=d_{n}\to0$.
We define
\begin{equation*}
v_{n}(r) = \dfrac{u_{n}(r)}{d_{n}}, \quad r\in\mathopen{[}0,R\mathclose{]},
\end{equation*}
and observe that $v_{n}(r)$ is a non-negative solution of the Neumann problem associated with
\begin{equation}\label{eq-v_n}
\Biggl{(} \dfrac{r^{N-1} v_{n}'}{\sqrt{1-(u_{n}')^{2}}}\Biggr{)}' + \vartheta_{n} \lambda r^{N-1} a(r) q(u_{n}(r)) v_{n} = 0,
\end{equation}
where we set $q(u) = g(u)/u$ for $u > 0$ and $q(0) = 0$. 
Integrating equation \eqref{eq-v_n} between $0$ and $r$ and dividing by $r^{N-1}$ we obtain that
\begin{equation*}
\dfrac{v_{n}'(r)}{\sqrt{1-u_{n}'(r)^{2}}} = -\dfrac{1}{r^{N-1}} \vartheta_{n} \lambda \int_{0}^{r} \xi^{N-1}a(\xi)q(u_{n}(\xi))v_{n}(\xi) \,\mathrm{d}\xi, \quad \text{for all $r \in \mathopen{]}0,R\mathclose{]}$.}
\end{equation*}
Passing to the absolute value we have
\begin{equation*}
|v_{n}'(r)| \leq \dfrac{|v_{n}'(r)|}{\sqrt{1-u_{n}'(r)^{2}}}
\leq \lambda \int_{0}^{R} |a(\xi)| |q(u_{n}(\xi))||v_{n}(\xi)| \,\mathrm{d}\xi, \quad \text{for all $r \in \mathopen{]}0,R\mathclose{]}$.}
\end{equation*}
Therefore, using the first condition in $(g_{0})$ and the fact that $\| v_{n} \|_{\infty} \leq 1$, we obtain that $v_{n}' \to 0$ uniformly.
As a consequence, $v_{n} \to 1$ uniformly in $\mathopen{[}0,R\mathclose{]}$, since
\begin{equation*}
|v_{n}(r) - 1 | = |v_{n}(r) - v_{n}(\hat{\eta}_{n}) | \leq \int_{0}^{R} | v_{n}'(\xi) | \,\mathrm{d}\xi, \quad \text{for all $r \in \mathopen{[}0,R\mathclose{]}$,}
\end{equation*}
where $\hat{\eta}_{n}\in \mathopen{[}0,R\mathclose{]}$ is such that $u_{n}(\hat{\eta}_{n}) = \|u_{n}\|_{\infty} = d_{n}$.
An integration of equation \eqref{eq-v_n} in $\mathopen{[}0,R\mathclose{]}$ gives
\begin{equation*}
\int_{0}^{R} r^{N-1} a(r) g(u_{n}(r))\,\mathrm{d}r = 0
\end{equation*}
and hence
\begin{equation*}
\int_{0}^{R} r^{N-1} a(r) g(d_{n})\,\mathrm{d}r + \int_{0}^{R} r^{N-1} a(r)\bigl{[}g(d_{n} v_{n}(r)) - g(d_{n})\bigr{]}\,\mathrm{d}r=0.
\end{equation*}
Dividing by $g(d_{n}) > 0$, we have
\begin{equation*}
0 < - \int_{0}^{R} r^{N-1} a(r)\,\mathrm{d}r \leq R^{N-1}\|a\|_{L^{1}} \sup_{r\in \mathopen{[}0,R\mathclose{]}}\biggl{|}\dfrac{g(d_{n} v_{n}(r))}{g(d_{n})} - 1\biggr{|}.
\end{equation*}
Using the second condition in $(g_{0})$ and recalling that $v_{n} \to 1$ uniformly, we find a contradiction. The claim is thus proved and we can fix $d^{*}\in\mathopen{[}0,\delta^{*}\mathclose{[}$.

Finally, condition $(H_{3})$ of Lemma~\ref{lem-deg1} is trivially satisfied for $d=d^{*}$, and therefore
\begin{equation}\label{deg-B0d*}
\mathrm{deg}_{\mathrm{LS}}(\mathrm{Id}-T_{1,0},B(0,d^{*})) = 1.
\end{equation}

\subsubsection*{Step~4. Fixing the constant $D^{*}$ and computation of the degree in $B(0,D^{*})$.}
First, we claim that there exists $D^{*}>\delta^{*}$ such that, for every $\vartheta\in \mathopen{]}0,1\mathclose{]}$, every non-negative solution $u(r)$ of \eqref{eq-lem-deg1} satisfies $\|u\|_{\infty} < D^{*}$.

We assume, by contradiction, that there exists a sequence $\{u_{n}\}_{n}$ of non-negative solutions of \eqref{eq-lem-deg1} for $\vartheta=\vartheta_{n}$ and $\lambda=\lambda_{n}$ satisfying $\|u_{n}\|_{\infty}=D_{n}\to+\infty$. We proceed similarly to the previous step. We define
\begin{equation*}
v_{n}(r) = \dfrac{u_{n}(r)}{D_{n}}, \quad r\in\mathopen{[}0,R\mathclose{]},
\end{equation*}
which solves the Neumann problem associated with equation \eqref{eq-v_n}. Since $\|u_{n}'\|_{\infty} \leq 1$, we easily find $\|v_{n}'\|_{\infty} \to 0$ and, consequently, $v_{n} \to 1$ uniformly in $\mathopen{[}0,R\mathclose{]}$ (proceeding as shown in Step~3).
Integrating equation \eqref{eq-v_n} and dividing by $g(D_{n}) > 0$, we thus obtain
\begin{equation*}
0 < - \int_{0}^{R} r^{N-1}a(r)\,\mathrm{d}r \leq R^{N-1}\|a\|_{L^{1}} \sup_{r\in \mathopen{[}0,R\mathclose{]}}\biggl{|}\dfrac{g(D_{n} v_{n}(r))}{g(D_{n})} - 1\biggr{|}.
\end{equation*}
Using $(g_{\infty})$, a contradiction easily follows. The claim is thus proved and we can fix $D^{*}>\delta^{*}$.

Finally, condition $(H_{3})$ of Lemma~\ref{lem-deg1} is trivially satisfied for $d=D^{*}$, and therefore
\begin{equation}\label{deg-B0D*}
\mathrm{deg}_{\mathrm{LS}}(\mathrm{Id}-T_{1,0},B(0,D^{*})) = 1.
\end{equation}

\subsubsection*{Step~5. Concluding the proof.}
Starting from the formulas \eqref{deg-B0rho*}, \eqref{deg-B0d*}, \eqref{deg-B0D*} proved in the last three steps, we apply the additivity property of the coincidence degree to obtain
\begin{equation*}
\mathrm{deg}_{\mathrm{LS}}(\mathrm{Id}-T_{1,0},B(0,\delta^{*}) \setminus \overline{B(0,d^{*})}) = -1
\end{equation*}
and
\begin{equation*}
\mathrm{deg}_{\mathrm{LS}}(\mathrm{Id}-T_{1,0},B(0,D^{*}) \setminus \overline{B(0,\delta^{*})}) = 1.
\end{equation*}
As a consequence of the existence property of the degree, there exist two solutions $u_{s}\in B(0,\delta^{*}) \setminus \overline{B(0,d^{*})}$ and $u_{\ell}\in B(0,D^{*}) \setminus \overline{B(0,\delta^{*})}$ of problem \eqref{eq-phi}. Then, $u_{s}(r)$ and $u_{\ell}(r)$ satisfy 
\begin{equation*}
d^{*} < \| u_{s} \|_{\infty} < \delta^{*} < \| u_{\ell} \|_{\infty} < D^{*}.
\end{equation*}
By maximum principle arguments, both these solutions are non-negative and hence they solve \eqref{eq-main}.
It thus remains to show that they are positive, that is, $u(r) > 0$ for every $r \in [0,R]$, for both $u = u_s$ and $u = u_{\ell}$.

By contradiction, assume that $u(r_0) = 0$ for some $r_0 \in [0,R]$. Then $u'(r_0) = 0$, coming from the boundary conditions if $r_0 = 0$ and from the fact that $u(r)$ is non-negative if $r_0 \in (0,R]$.

If $r_0 = 0$, integrating the equation we find
\begin{equation*}
u(r) = - \int_{0}^r \varphi^{-1}\biggl{(}\dfrac{1}{\zeta^{N-1}} \int_{0}^{\zeta} \xi^{N-1}a(\xi)g(u(\xi))\,\mathrm{d}\xi \biggr{)} \,\mathrm{d}\zeta, \quad \text{for all $r \in \mathopen{[}0,R\mathclose{]}$,}
\end{equation*}
implying, since $|\varphi^{-1}(s)|\leq |s|$ for all $s \in \mathbb{R}$, that
\begin{equation*}
|u(r)| \leq \int_{0}^r \int_{0}^{\zeta} |a(\xi)| |g(u(\xi))|\,\mathrm{d}\xi \,\mathrm{d}\zeta 
\leq R \int_{0}^r |a(\xi)| |g(u(\xi))|\,\mathrm{d}\xi, 
\end{equation*}
for all $r \in \mathopen{[}0,R\mathclose{]}$.
Recalling that $g(u)/u \to 0$ for $u \to 0^{+}$, we finally obtain 
\begin{equation*}
|u(r)| \leq M R \int_{0}^r |a(\xi)| |u(\xi)|\,\mathrm{d}\xi , \quad \text{for all $r \in \mathopen{[}0,R\mathclose{]}$,}
\end{equation*}
where $M > 0$ is a suitable constant. By Gronwall's lemma, $u(r) = 0$ for every $r \in \mathopen{[}0,R\mathclose{]}$, a contradiction.

In the case $r_0 \in \mathopen{]}0,R\mathclose{]}$, the contradiction is reached again using the assumption $g(u)/u \to 0$ for $u \to 0^{+}$, which as well known implies
 (together with the smoothness of $\varphi^{-1}$)
that the only solution of the planar system
\begin{equation*}
u' = \varphi^{-1}\biggl{(} \dfrac{v}{r^{N-1}} \biggr{)}, \qquad v' = -r^{N-1} a(r) g(u),
\end{equation*}
satisfying the initial condition $(u(r_0),v(r_0)) = (0,0)$ is the trivial one.
\qed

\bibliographystyle{elsart-num-sort}
\bibliography{BoFe-biblio}

\end{document}